\documentclass[11pt,reqno]{amsart}

\usepackage[letterpaper,left=0.80in,right=0.80in,top=0.95in,bottom=0.95in]{geometry}
\usepackage[T1]{fontenc}
\usepackage[utf8]{inputenc}
\usepackage{lmodern}
\usepackage{microtype}
\usepackage{amssymb,mathtools}
\usepackage{booktabs,array}
\usepackage{enumitem}
\usepackage{xcolor}
\usepackage{hyperref}

\hypersetup{
  colorlinks=true,
  linkcolor=blue!55!black,
  citecolor=blue!55!black,
  urlcolor=blue!55!black,
  pdftitle={Positive rational series for reciprocal powers of Catalan's constant and Dirichlet beta values},
  pdfauthor={Narendra Bhandari}
}

\allowdisplaybreaks[2]
\numberwithin{equation}{section}
\newtheorem{theorem}{Theorem}[section]
\newtheorem{proposition}[theorem]{Proposition}
\newtheorem{lemma}[theorem]{Lemma}
\newtheorem{corollary}[theorem]{Corollary}
\newtheorem{conjecture}[theorem]{Conjecture}
\theoremstyle{definition}
\newtheorem{definition}[theorem]{Definition}
\theoremstyle{remark}
\newtheorem{remark}[theorem]{Remark}

\DeclareMathOperator{\Ti}{Ti}
\DeclareMathOperator{\RePart}{Re}
\newcommand{\dd}{\,\mathrm{d}}
\newcommand{\e}{\mathrm{e}}
\newcommand{\Gconst}{G}
\newcommand{\poch}[2]{\left(#1\right)_{#2}}

\title[Reciprocal powers of Catalan's constant]
{Positive rational series for reciprocal powers of Catalan's constant
and Dirichlet beta values}
\author{Narendra Bhandari}
\address{Department of Mathematics, University of North Texas, Denton, Texas 76201, USA}
\email{narendra.bhandari@unt.edu}
\email{narenbhandari04@gmail.com}
\date{July 2026}



\begin{document}

\begin{abstract}
Let
\(
\beta(s)=\sum_{k=0}^{\infty}{(-1)^k}/{(2k+1)^s}\) and \(
\Gconst=\beta(2),
\)
where \(\Gconst\) is Catalan's constant.  We develop two families of
positive series for reciprocal powers \(\beta(s)^{-r}\), where
\(r\geq1\) is an integer.  The first
comes from the classical Euler transformation and is evaluated at
\(1/2\); it is valid for real \(s>0\).  A second transformation, valid
for \(s\geq2\), produces a faster series evaluated at \(1/3\).  We
obtain finite-sum and integral formulas for the base coefficients,
together with a positive recurrence and a composition formula for
all positive integral reciprocal powers.  When \(s\) is an integer, all
coefficients are rational.  We also give explicit remainder estimates
and determine the exact root-convergence rates of the two families.
As applications, we obtain positive rational series for every
reciprocal power of Catalan's constant and for reciprocal powers of
\(\pi\) arising from odd values of the Dirichlet beta function.
\end{abstract}
\maketitle
\vspace{-0.9\baselineskip}
\section{Introduction}
\label{sec:introduction}

\subsection{Background and purpose}

The Dirichlet beta function is defined by
\begin{equation}\label{eq:G-definition}
 \beta(s)
 =
\sum_{n=0}^{\infty}\frac{(-1)^n}{(2n+1)^s},
 \qquad s>0,
 \qquad
 \Gconst=\beta(2).
\end{equation}
The value \(\Gconst\) is Catalan's constant.  It has a simple
alternating-series definition, but its arithmetic nature is still not
understood: in particular, it is not known whether \(\Gconst\) is
irrational.

Catalan studied transformations of related series and definite
integrals \cite{Catalan1867}.  The inverse tangent integral is discussed
in \cite{Ramanujan1915,Lewin1981,Campbell2022}.  Accelerated series and
rational approximations for Catalan's constant are studied in
\cite{Bradley1999}, while continued-fraction representations appear in
\cite{BowmanMcLaughlin2002,Zudilin2003}.  Important arithmetic results for Catalan's constant and related beta values were obtained by Rivoal and
Zudilin \cite{RivoalZudilin2003}, Rivoal \cite{Rivoal2006},
Nesterenko \cite{Nesterenko2016}, and Krattenthaler and Zudilin
\cite{KrattenthalerZudilin2019}.  Zudilin
\cite[Thm.~1]{Zudilin2019} proved that at least one of
\(\beta(2),\beta(4),\ldots,\beta(12)\) is irrational.  These results
do not, however, settle the irrationality of \(\Gconst\), and that is
not the aim of this paper.

Here we study reciprocal powers
\[
 \beta(s)^{-r},\qquad r\geq1,
\]
where \(r\) is an integer.
The Euler product for \(L(s,\chi_4)=\beta(s)\) already gives a
Dirichlet-series expansion for these quantities.  We take a different
approach.  Our goal is to obtain ordinary power series whose terms are
positive and, when \(s\) is an integer, rational.  The coefficients are
given by finite formulas and can be generated efficiently by a simple
recurrence.
The first construction is based on Euler's transformation of an
alternating series; a classical treatment is given by Hardy
\cite[Ch.~VIII]{Hardy1949}.  Related acceleration methods were
studied by Wynn \cite{Wynn1971} and Cohen et al.~\cite{CohenRVZ2000}.
The reciprocal of a power series has its own long history, including
work of Kaluza \cite{Kaluza1928}, Lamperti \cite{Lamperti1958}, and
Baricz et al.~\cite{BariczVestiVuorinen2011}.

\subsection{Main results}

The idea is straightforward.  We begin with a generating function for
the beta series and apply two fractional-linear substitutions.  The
classical Euler substitution \(t=w/(1-w)\) leads to a positive series
evaluated at \(w=1/2\) for every real \(s>0\).  A second substitution,
\begin{equation}\label{eq:transform-intro}
 t=\frac{2u}{1-u},
\end{equation}
leads to a positive series evaluated at \(u=1/3\) when \(s\geq2\).
The resulting representations are
\[
 \frac1{\beta(s)^r}
 =
 \sum_{n=0}^{\infty}\frac{b_{s,r,n}}{2^n},
 \qquad
 \frac1{\beta(s)^r}
 =
 \sum_{n=0}^{\infty}\frac{c_{s,r,n}}{3^n},
\]
with positive coefficients.  Their exact root-convergence rates are
\(1/2\) and \(1/3\), respectively.  We also give explicit remainder
bounds, finite-sum and integral formulas for the base coefficients,
and a recurrence and composition formula for their reciprocal powers.
Two applications are especially simple to state.  For every integer
\(r\geq1\), Corollary~\ref{cor:all-G} gives
\begin{equation*}
\begin{aligned}
 \frac1{\Gconst^r}
 &=1+
 \sum_{n=1}^{\infty}\frac1{3^n}
 \sum_{j=1}^{n}\binom{r+j-1}{j}
\sum_{\substack{n_1+\cdots+n_j=n\\n_i\geq1}}
 \prod_{\ell=1}^{j}
 \left[
 \sum_{k=1}^{n_\ell}
 \frac{(-1)^{k+1}2^k}{(2k+1)^2}
 \binom{n_\ell-1}{k-1}
 \right],
\end{aligned}
\end{equation*}
and Corollary~\ref{cor:pi1-ordinary} gives
\begin{equation*}
\begin{aligned}
 \frac{4^r}{\pi^r}
 &=
 1+\sum_{n=1}^{\infty}\frac1{2^n}
 \sum_{j=1}^{n}\binom{r+j-1}{j}
\sum_{\substack{n_1+\cdots+n_j=n\\n_i\geq1}}
 \prod_{\ell=1}^{j}
 \left[
 \sum_{k=1}^{n_\ell}
 \frac{(-1)^{k+1}}{2k+1}
 \binom{n_\ell-1}{k-1}
 \right].
\end{aligned}
\end{equation*}
Every term after the initial \(1\) is a positive rational number.
The first formula is proved in Section~\ref{sec:even-beta}, and the
second in Section~\ref{sec:odd-beta}.  The method is Euler-type rather
than modular or hypergeometric in the sense of Ramanujan-type series.

\section{Preliminary representations of beta values}
\label{sec:standard-dirichlet}

\subsection{A classical Dirichlet series}

Before beginning the positive power-series construction, we recall a
standard Dirichlet-series identity.  Apostol gives the case \(r=1\)
explicitly \cite[\S11.4, Example~3]{Apostol1976}, and the general
positive-integral case appears in Ng and Toma
\cite[Eq.~(5.3)]{NgToma2024}.  With \(\chi=\chi_4\) and \(s=2\), the
identity becomes a series for \(\Gconst^{-r}\).  This classical
representation will serve only as a point of comparison for the
positive series developed below.

Let \(\chi_4\) be the primitive character modulo \(4\):
\begin{equation}\label{eq:chi4}
 \chi_4(n)=
 \begin{cases}
   0,&2\mid n,\\
   1,&n\equiv1\pmod4,\\
  -1,&n\equiv3\pmod4.
 \end{cases}
\end{equation}
Then
\[
 \beta(s)=L(s,\chi_4)
 =\sum_{n=1}^{\infty}\frac{\chi_4(n)}{n^s}.
\]
Let \(\mu\) denote the M\"obius function.  For arithmetic functions
\(f\) and \(g\), their Dirichlet convolution is defined by
\[
 (f*g)(n)=\sum_{d\mid n}f(d)g(n/d).
\]
Define the \(r\)-fold Dirichlet convolution
\begin{equation}\label{eq:mu-r}
 \mu_r=\underbrace{\mu*\mu*\cdots *\mu}_{r\ {\rm factors}}.
\end{equation}

\begin{theorem}
\label{thm:standard-L}
For every integer \(r\geq1\) and every real \(s>1\),
\begin{equation}\label{eq:standard-L-series}
 \frac1{\beta(s)^r}
 =\sum_{n=1}^{\infty}\frac{\chi_4(n)\mu_r(n)}{n^s}.
\end{equation}
The series in \eqref{eq:standard-L-series} converges absolutely.
Moreover, \(\mu_r\) is multiplicative and
\begin{equation}\label{eq:mu-r-prime}
 \mu_r(p^a)=
 \begin{cases}
  (-1)^a\binom{r}{a},&0\leq a\leq r,\\
  0,&a>r.
 \end{cases}
\end{equation}
\end{theorem}

\medskip
\noindent A proof is given in Appendix~\ref{app:standard-L-proof}.

\begin{corollary}
For every integer \(r\geq1\),
\begin{equation}\label{eq:standard-G}
 \frac1{\Gconst^r}
 =\sum_{n=1}^{\infty}
   \frac{\chi_4(n)\mu_r(n)}{n^2}.
\end{equation}
\end{corollary}

\begin{proof}
The result follows immediately by taking \(s=2\) in
\eqref{eq:standard-L-series} and using
\(\beta(2)=\Gconst\).
\end{proof}

For \(r=1\) and \(r=2\), this gives the following explicit
Dirichlet series for reciprocal powers of Catalan's constant:
\begin{align}
 \sum_{n=1}^{\infty}
   \frac{\chi_4(n)\mu_1(n)}{n^2}
 &=1+\frac1{3^2}-\frac1{5^2}+\frac1{7^2}
   +\frac1{11^2}-\frac1{13^2}-\frac1{15^2}-\cdots
   =\frac1{\Gconst},
 \label{eq:standard-G1}\\
 \sum_{n=1}^{\infty}
   \frac{\chi_4(n)\mu_2(n)}{n^2}
 &=1+\frac2{3^2}-\frac2{5^2}+\frac2{7^2}
   +\frac1{9^2}+\frac2{11^2}-\frac2{13^2}
   -\frac4{15^2}-\cdots
   =\frac1{\Gconst^2}.
 \label{eq:standard-G2}
\end{align}
These formulas provide a useful comparison.  Their signs do not follow a
fixed alternating pattern, and their terms decay algebraically rather
than geometrically.  The construction below replaces them with
positive rational terms and geometric convergence.

\subsection{An integral representation}

Both transformations rest on a simple integral representation.  For
real \(s>0\), define
\begin{equation}\label{eq:F-definition}
 F_s(t)=\sum_{k=0}^{\infty}\frac{(-1)^kt^k}{(2k+1)^s},
 \qquad |t|<1.
\end{equation}

\begin{proposition}\label{prop:F-integral}
For each fixed real \(s>0\), the function \(F_s(t)\) extends
analytically to the half-plane
\(\RePart t>-1\). Then
\begin{equation}\label{eq:F-integral}
 F_s(t)
 =
 \frac1{\Gamma(s)}
 \int_0^1
 \frac{(-\log x)^{s-1}}{1+tx^2}\dd x.
\end{equation}
In particular,
\(
 F_s(1)=\beta(s).
\)
\end{proposition}

\begin{proof}
After the change of variable \(y=-\log x\), we note that
\begin{equation}\label{eq:moment}
 \frac1{(2k+1)^s}
 =
 \frac1{\Gamma(s)}
 \int_0^1x^{2k}(-\log x)^{s-1}\dd x.
\end{equation}
For \(|t|<1\), substitute \eqref{eq:moment} into
\eqref{eq:F-definition}.  Uniform convergence of the geometric series
on \(0\leq x\leq1\) allows integration term by term, and hence
\[
 \begin{aligned}
 F_s(t)
 &=
 \frac1{\Gamma(s)}
 \int_0^1
 (-\log x)^{s-1}
 \sum_{k=0}^{\infty}(-tx^2)^k\dd x=
 \frac1{\Gamma(s)}
 \int_0^1
 \frac{(-\log x)^{s-1}}{1+tx^2}\dd x.
 \end{aligned}
\]

The same integral extends the function to the larger half-plane.  Let
\(K\) be a compact subset of \(\RePart t>-1\).  Compactness gives a
constant \(\delta_K>0\) for which
\[
 \RePart(1+tx^2)
 =1+x^2\RePart t
 \geq\delta_K
 \qquad
 (t\in K,\ 0\leq x\leq1).
\]
Consequently,
\(
 |1+tx^2|\geq\delta_K
\)
and hence
\[
 \left|
 \frac{(-\log x)^{s-1}}{1+tx^2}
 \right|
 \leq
 \frac{(-\log x)^{s-1}}{\delta_K}.
\]
The function on the right is integrable because
\begin{align}\label{eq-gamma}
    \int_0^1(-\log x)^{s-1}\dd x=\Gamma(s)<\infty.
\end{align}
Thus, on every compact subset of \(\RePart t>-1\), the integrand is
bounded by an integrable function.  Because it is analytic in \(t\) for
each \(0<x<1\), the integral defines an analytic function there.  On
\(|t|<1\) it agrees with the original power series, so it is the
required analytic continuation.

It remains to identify the value at \(t=1\).  For \(0<\rho<1\), the
power series and the integral agree at \(t=\rho\).  Since the
alternating series for \(\beta(s)\) converges for every \(s>0\),
Abel's theorem gives
\[
 \lim_{\rho\to1^-}
 \sum_{k=0}^{\infty}
 \frac{(-1)^k\rho^k}{(2k+1)^s}
 =
 \beta(s).
\]
On the other hand, since
\[
 0<
 \frac{(-\log x)^{s-1}}{1+\rho x^2}
 \leq
 (-\log x)^{s-1},
\]
dominated convergence gives
\[
 \lim_{\rho\to1^-}F_s(\rho)
 =
 \frac1{\Gamma(s)}
 \int_0^1
 \frac{(-\log x)^{s-1}}{1+x^2}\dd x
 =
 F_s(1).
\]
The two limits agree, and hence \(F_s(1)=\beta(s)\).
\end{proof}
\begin{remark}\label{rem:positive-real-part}
The integral representation also shows directly that \(F_s(t)\) does not
vanish when \(\RePart t>-1\).  Indeed, both \(1+tx^2\) and its
reciprocal have positive real part for \(0\leq x\leq1\).  It follows that
\begin{equation}\label{eq:F-positive-real}
 \RePart F_s(t)>0\qquad(\RePart t>-1).
\end{equation}
This will later show that \(H_s(w)\) and \(A_s(u)\) do not vanish in
the unit disk; see Proposition~\ref{prop:zero-free}.
\end{remark}

\section{Two Euler transformations}
\label{sec:transforms}

\subsection{The ordinary Euler transformation}

For the first transformation, use the classical Euler substitution
\(
 t=\frac{w}{1-w}.
\)
Here \(t=1\) corresponds to \(w=1/2\).  Define
\begin{equation}\label{eq:H-definition}
 H_s(w)=F_s\!\left(\frac{w}{1-w}\right).
\end{equation}
\begin{theorem}\label{thm:H-coefficients}
For real \(s>0\) and \(|w|<1\),
\begin{equation}\label{eq:H-expansion}
 H_s(w)=1-\sum_{n=1}^{\infty}q_{s,n}w^n,
\end{equation}
where
\begin{align}
 q_{s,n}
 &=
 \sum_{k=1}^{n}
 \frac{(-1)^{k+1}}{(2k+1)^s}
 \binom{n-1}{k-1},                       \label{eq:q-finite}\\
 &=
 \frac1{\Gamma(s)}
 \int_0^1
 x^2(1-x^2)^{n-1}
 (-\log x)^{s-1}\dd x.                 \label{eq:q-integral}
\end{align}
In particular,
\(
 q_{s,n}>0
 \) for \(s>0,\ n\geq1.
\)
\end{theorem}

\begin{proof}
Substituting \eqref{eq:H-definition} into the integral representation
\eqref{eq:F-integral} gives
\[
 \begin{aligned}
 H_s(w)
 &=
 \frac1{\Gamma(s)}
 \int_0^1
 \frac{(1-w)(-\log x)^{s-1}}
 {1-w+x^2w}\dd x\\
 &=
 \frac1{\Gamma(s)}
 \int_0^1
 \frac{(1-w)(-\log x)^{s-1}}
 {1-(1-x^2)w}\dd x.
 \end{aligned}
\]

For \(|w|<1\), expanding the denominator as a geometric series gives
\[
 \begin{aligned}
 \frac{1-w}{1-(1-x^2)w}
 &=
 1-\frac{x^2w}{1-(1-x^2)w}=
 1-x^2w
 \sum_{m=0}^{\infty}(1-x^2)^mw^m\\
 &=
 1-\sum_{n=1}^{\infty}
 x^2(1-x^2)^{n-1}w^n.
 \end{aligned}
\] 
To justify integration term by term, fix \(|w|\leq\rho<1\).  Then
\[
 \begin{aligned}
 \sum_{n=1}^{\infty}
 \left|
 x^2(1-x^2)^{n-1}w^n
 \right|
 &\leq
 \sum_{n=1}^{\infty}\rho^n=
 \frac{\rho}{1-\rho}.
 \end{aligned}
\]
Together with \eqref{eq-gamma}, this bound justifies integration term by
term for \(|w|\leq\rho<1\).  We therefore obtain

\[
 \begin{aligned}
 H_s(w)
 &=
 \frac1{\Gamma(s)}
 \int_0^1(-\log x)^{s-1}\dd x-
 \sum_{n=1}^{\infty}
 \left[
 \frac1{\Gamma(s)}
 \int_0^1
 x^2(1-x^2)^{n-1}
 (-\log x)^{s-1}\dd x
 \right]w^n.
 \end{aligned}
\]
Comparison with \eqref{eq:H-expansion}, using
\eqref{eq-gamma}, now yields
\[
 q_{s,n}
 =
 \frac1{\Gamma(s)}
 \int_0^1
 x^2(1-x^2)^{n-1}
 (-\log x)^{s-1}\dd x,
\]
This is \eqref{eq:q-integral}.  To recover the finite formula, expand
\[
 (1-x^2)^{n-1}
 =
 \sum_{j=0}^{n-1}
 (-1)^j\binom{n-1}{j}x^{2j}.
\]
Substituting this finite expansion into \eqref{eq:q-integral} and
using \eqref{eq:moment} gives
\[
 \begin{aligned}
 q_{s,n}
 &=
 \frac1{\Gamma(s)}
 \sum_{j=0}^{n-1}
 (-1)^j\binom{n-1}{j}
 \int_0^1
 x^{2j+2}(-\log x)^{s-1}\dd x\\
 &=
 \sum_{j=0}^{n-1}
 \frac{(-1)^j}{(2j+3)^s}
 \binom{n-1}{j}.
 \end{aligned}
\]
Relabeling \(k=j+1\) gives \eqref{eq:q-finite}.  Positivity is
immediate from \eqref{eq:q-integral}: every factor in the integrand is
positive on \(0<x<1\), and \(\Gamma(s)>0\).
\end{proof}

\begin{remark}\label{rem:euler-beta}
At \(w=1/2\), we have
\(H_s(1/2)=F_s(1)=\beta(s)\).  Thus, \eqref{eq:H-expansion} becomes
\[
 \beta(s)
 =
 1-\sum_{n=1}^{\infty}\frac{q_{s,n}}{2^n}=1-\sum_{n=1}^{\infty}\sum_{k=1}^{n}
 \frac{(-1)^{k+1}}{2^n(2k+1)^s}
 \binom{n-1}{k-1}.
\]
This is the classical Euler transform of the alternating beta series.
The second transformation below has the same general shape, but moves
the evaluation point from \(1/2\) to \(1/3\).
\end{remark}

\subsection{Closed forms for the ordinary coefficients}

\begin{proposition}\label{prop:q-beta}
Let \(s\geq1\) be an integer.  Then, for every \(n\geq1\),
\begin{equation}\label{eq:q-beta-derivative}
 q_{s,n}
 =\frac{(-1)^{s-1}}{2^s\Gamma(s)}
 \left.
 \frac{\partial^{s-1}}{\partial a^{s-1}}B(a,n)
 \right|_{a=3/2}=\frac{B(3/2,n)}{2^s\Gamma(s)}
 Y_{s-1}(\Delta_{1,n},\ldots,\Delta_{s-1,n}).
\end{equation}
In this formula, \(Y_m\) is the complete exponential Bell polynomial
and
\begin{equation}\label{eq:Dj}
 \Delta_{j,n}=(j-1)!
 \left\{2^j\bigl(H_{2n+1}^{(j)}-1\bigr)-H_n^{(j)}\right\}.
\end{equation}
Here
\(
 H_N^{(j)}=\sum_{k=1}^{N}\frac1{k^j}
\)
denotes the generalized harmonic number of order \(j\), and we write
\(H_N=H_N^{(1)}\).  When \(s=1\), the convention \(Y_0=1\) is used.
\end{proposition}

\begin{proof}
The substitution \(y=x^2\) in \eqref{eq:q-integral} gives
\[
 q_{s,n}
 =
 \frac1{2^s\Gamma(s)}
 \int_0^1y^{1/2}(1-y)^{n-1}
 (-\log y)^{s-1}\dd y.
\]
Using Euler's beta integral
\cite[Eq.~(1.1.12)]{AndrewsAskeyRoy1999}
and differentiating \(s-1\) times with respect to its first parameter,
we obtain
\[
 q_{s,n}
 =
 \frac{(-1)^{s-1}}{2^s\Gamma(s)}
 \left.
 \frac{\partial^{s-1}}{\partial a^{s-1}}B(a,n)
 \right|_{a=3/2}.
\]
This proves the first equality in
\eqref{eq:q-beta-derivative}.  For the second, write
\(
 L(a)=\log B(a,n).
\)
Since
\[
 B(a,n)
 =
 \frac{(n-1)!}{a(a+1)\cdots(a+n-1)},
\]
we have
\[
 L^{(j)}(a)
 =
 (-1)^j(j-1)!
 \sum_{\ell=0}^{n-1}\frac1{(a+\ell)^j}.
\]
At \(a=3/2\), separating the odd and even terms gives
\[
 (j-1)!
 \sum_{\ell=0}^{n-1}\frac1{(\ell+3/2)^j}
 =
 (j-1)!
 \left\{
 2^j\bigl(H_{2n+1}^{(j)}-1\bigr)-H_n^{(j)}
 \right\}
 =\Delta_{j,n}.
\]
Hence
\(L^{(j)}(3/2)=(-1)^j\Delta_{j,n}\).  The Fa\`a di Bruno formula in Bell-polynomial form
\cite[Ch.~III, Thm.~C]{Comtet1974}, together with the homogeneity of
the Bell polynomials, gives
\[
 (-1)^{s-1}
 \left.
 \frac{\partial^{s-1}}{\partial a^{s-1}}B(a,n)
 \right|_{a=3/2}
 =
 B(3/2,n)
 Y_{s-1}(\Delta_{1,n},\ldots,\Delta_{s-1,n}).
\]
Inserting this identity into the first formula gives the result.
\end{proof}

\begin{corollary}\label{cor:beta-first-three}
The ordinary Euler transform yields the following formulas for
\(s=1,2,3\):
\begin{align*}
 & \frac{\pi}{4}
 =
 1-\sum_{n=1}^{\infty}
 \frac{2^{n-1}}
 {n(2n+1)\binom{2n}{n}}, \\&
 \Gconst
 =
 1-\sum_{n=1}^{\infty}
 \frac{
 2^{n-2}\bigl(2H_{2n+1}-H_n-2\bigr)}
 {n(2n+1)\binom{2n}{n}},\\&
 \frac{\pi^3}{32}
 =
 1-\sum_{n=1}^{\infty}
 \frac{
 2^{n-4}
 \left[
 \bigl(2H_{2n+1}-H_n-2\bigr)^2
 +4H_{2n+1}^{(2)}-H_n^{(2)}-4
 \right]}
 {n(2n+1)\binom{2n}{n}}.
\end{align*}
\end{corollary}
\begin{proof}
First note the beta-function identity
\[
 B(3/2,n)
 =\frac{4^n}{n(2n+1)\binom{2n}{n}}
\]
follows from \(B(a,n)=\Gamma(a)\Gamma(n)/\Gamma(a+n)\).
Proposition~\ref{prop:q-beta} may now be applied with
\[
 Y_0=1,\qquad
 Y_1(x_1)=x_1,\qquad
 Y_2(x_1,x_2)=x_1^2+x_2.
\]
For \(s=1,2,3\), substituting the resulting coefficients into
Remark~\ref{rem:euler-beta} gives the three formulas.
\end{proof}

\subsection{The accelerated transformation}

The second substitution is chosen for the half-plane of analyticity
in Proposition~\ref{prop:F-integral}.  Indeed,
\eqref{eq:transform-intro} maps \(|u|<1\) onto \(\RePart t>-1\), since
\(
 1+t=\frac{1+u}{1-u},
\)
and it sends \(u=1/3\) to \(t=1\).  Set
\begin{equation}\label{eq:A-definition}
 A_s(u)=F_s\!\left(\frac{2u}{1-u}\right).
\end{equation}

\begin{theorem}\label{thm:p-coefficients}
For real \(s>0\) and \(|u|<1\),
\begin{equation}\label{eq:A-expansion}
 A_s(u)=1-\sum_{n=1}^{\infty}p_{s,n}u^n,
\end{equation}
where
\begin{align}
 p_{s,n}
 &=\sum_{k=1}^{n}
   \frac{(-1)^{k+1}2^k}{(2k+1)^s}
   \binom{n-1}{k-1},                     \label{eq:p-finite}\\
 &=\frac2{\Gamma(s)}
   \int_0^1x^2(1-2x^2)^{n-1}
   (-\log x)^{s-1}\dd x.                 \label{eq:p-integral}
\end{align}

\end{theorem}

\begin{proof}
The calculation parallels the proof of
Theorem~\ref{thm:H-coefficients}.  Substitute \(t=2u/(1-u)\) into
\eqref{eq:F-integral} and use
\[
 \frac{1-u}{1+(2x^2-1)u}
 =
 1-\sum_{n=1}^{\infty}
 2x^2(1-2x^2)^{n-1}u^n
\]
Integrating term by term gives \eqref{eq:A-expansion} and
\eqref{eq:p-integral}.  Expanding \((1-2x^2)^{n-1}\) and then applying
\eqref{eq:moment} yields \eqref{eq:p-finite}.
\end{proof}

Unlike \(q_{s,n}\), the coefficient \(p_{s,n}\) is not visibly
positive from its integral formula: when \(n\) is even, the integrand
changes sign.  The following comparison resolves this difficulty.

\begin{lemma}\label{lem:log-inequality}
If \(0<y<1/2\) and \(m\geq1\) is real, then
\begin{equation}\label{eq:f-inequality}
 \sqrt y\,[-\log y]^m
 >
 \sqrt{1-y}\,[-\log(1-y)]^m.
\end{equation}
\end{lemma}

\begin{proof}
Put
\(
 a=\sqrt{\frac{1-y}{y}}>1.
\)
Then
\(
 -\log y=\log(1+a^2),\,
 -\log(1-y)=\log(1+a^{-2}).
\)
We first establish
\begin{equation}\label{eq:h-positive}
 \log(1+a^2)>a\log(1+a^{-2}).
\end{equation}
Let
\(
 h(a)=\log(1+a^2)-a\log(1+a^{-2}).
\)
We have \(h(1)=0\), and
\[
 h'(a)
 =\frac{2(a+1)}{1+a^2}-\log(1+a^{-2}).
\]
The elementary inequality \(\log(1+z)<z\), \(z>0\), gives
\[
 h'(a)>
 \frac{2(a+1)}{1+a^2}-\frac1{a^2}>0
 \qquad(a>1).
\]
Hence \eqref{eq:h-positive} holds.  Now let
\[
 R=\frac{-\log y}{-\log(1-y)},
\]
then \(R>a>1\).  Since \(m\geq1\), we have \(R^m\geq R>a\).
Therefore,
\[
 \frac{\sqrt y\,[-\log y]^m}
 {\sqrt{1-y}\,[-\log(1-y)]^m}
 =\frac1a
 \left(\frac{-\log y}{-\log(1-y)}\right)^m>1,
\]
and \eqref{eq:f-inequality} follows.
\end{proof}

\begin{theorem}\label{thm:p-positive}
For every real \(s\geq2\) and every \(n\geq1\),
\( p_{s,n}>0\).
Equivalently, every non-constant Taylor coefficient of \(A_s(u)\) is
strictly negative.
\end{theorem}

\begin{proof}
Let \(m=s-1\geq1\).  After the substitution \(y=x^2\), divide out the positive factor
\(2^{1-s}/\Gamma(s)\).  We are left with
\begin{equation}\label{eq:p-y-integral}
 \int_0^1
 y^{1/2}(1-2y)^{n-1}(-\log y)^m\dd y.
\end{equation}
For odd \(n\), the integrand is nonnegative and not identically zero.
For even \(n\), split the integral at \(1/2\) and replace \(y\) by
\(1-y\) in the second half.  This gives
\[
 \int_0^{1/2}(1-2y)^{n-1}
 \left\{
 \sqrt y\,(-\log y)^m
 -\sqrt{1-y}\,[-\log(1-y)]^m
 \right\}\dd y.
\]
Lemma~\ref{lem:log-inequality} makes the bracket positive on
\(0<y<1/2\), while the remaining factor is also positive.  Thus,
\(p_{s,n}>0\).
\end{proof}

\begin{remark}\label{rem:s-threshold}
The condition \(s\geq2\) is meaningful.  At \(s=1\), for example,
\[
 [u^2]A_1(u)=\frac2{15}>0,
\]
so the non-constant coefficients no longer have one sign.
\end{remark}

\section{Reciprocal powers and their coefficients}
\label{sec:reciprocal-coefficients}

The two transformations now fit a common framework.  In each case the
transformed function has the form \(1-D(z)\), where
\begin{equation}\label{eq:common-shape}
 D(z)=\sum_{n=1}^{\infty}d_nz^n,\qquad d_n>0.
\end{equation}
For the ordinary Euler transform, \(d_n=q_{s,n}\).  For the
accelerated transform, \(d_n=p_{s,n}\), with \(s\geq2\).

\begin{definition}
For an integer \(r\geq1\), define
\begin{equation}\label{eq:C-definition}
 C_r(z)
 =\frac{1}{(1-D(z))^r}
 =\left(1-\sum_{n=1}^{\infty}d_nz^n\right)^{-r}
 =\sum_{n=0}^{\infty}c_n^{(r)}z^n.
\end{equation}
\end{definition}

\begin{theorem}\label{thm:recurrence}
Let \(r\geq1\) be an integer.  The coefficients in
\eqref{eq:C-definition} satisfy
\begin{equation}\label{eq:recurrence}
 c_0^{(r)}=1,\qquad
 c_n^{(r)}
 =\frac1n\sum_{k=1}^{n}
 \bigl[n+(r-1)k\bigr]d_kc_{n-k}^{(r)}
 \quad(n\geq1).
\end{equation}
In particular, \(c_n^{(r)}>0\).  If all \(d_k\) are rational, then
every \(c_n^{(r)}\) is rational.
\end{theorem}

\begin{proof}
Because \(D(0)=0\), we have \(C_r(0)=1\), so
\(c_0^{(r)}=1\).  Differentiating \(C_r=(1-D)^{-r}\) gives
\[
 C_r'=rD'(1-D)^{-r-1}.
\]
Multiplying both sides by \(1-D\) gives
\(
 (1-D)C_r'=rD'C_r.
\)
Now,
\[
 C_r'(z)=\sum_{m=1}^{\infty}m c_m^{(r)}z^{m-1},
 \qquad
 D'(z)=\sum_{k=1}^{\infty}k d_kz^{k-1}.
\]
For \(n\geq1\), the coefficient of \(z^{n-1}\) in \(C_r'\) is
\(n c_n^{(r)}\), whereas the corresponding coefficient in \(DC_r'\)
is
\[
 \sum_{k=1}^{n-1}
 d_k(n-k)c_{n-k}^{(r)}.
\]
Thus, the coefficient on the left-hand side is
\[
 n c_n^{(r)}
 -
 \sum_{k=1}^{n-1}
 (n-k)d_kc_{n-k}^{(r)}.
\]
On the right-hand side, the coefficient of \(z^{n-1}\) in \(D'C_r\)
is
\[
 \sum_{k=1}^{n}
 k d_kc_{n-k}^{(r)}.
\]
Equating these coefficients gives
\[
 n c_n^{(r)}
 -
 \sum_{k=1}^{n-1}
 (n-k)d_kc_{n-k}^{(r)}
 =
 r\sum_{k=1}^{n}
 k d_kc_{n-k}^{(r)}.
\]
The term \(k=n\) may be included in the first sum because its factor
\(n-k\) is zero.  Hence
\[
 \begin{aligned}
 n c_n^{(r)}
 &=
 \sum_{k=1}^{n}
 \bigl[(n-k)+rk\bigr]d_kc_{n-k}^{(r)}=
 \sum_{k=1}^{n}
 \bigl[n+(r-1)k\bigr]d_kc_{n-k}^{(r)}.
 \end{aligned}
\]
Division by \(n\) yields \eqref{eq:recurrence}.  Positivity now follows
by induction: \(c_0^{(r)}=1\), and every term on the right-hand side is
positive when the preceding coefficients are positive.  The same
recurrence also preserves rationality, since it uses only finite sums,
products, and division by the integer \(n\).
\end{proof}

\begin{theorem}
\label{thm:nonrecursive}
Let \(r\geq1\) be an integer.  For \(n\geq1\),
\begin{align}
 c_n^{(r)}
 &=
 \sum_{j=1}^{n}\binom{r+j-1}{j}
 \sum_{\substack{n_1+\cdots+n_j=n\\n_i\geq1}}
 d_{n_1}\cdots d_{n_j}, \label{eq:composition-form}\\
 &=
 \frac1{n!}\sum_{j=1}^{n}
 \poch{r}{j}
 B_{n,j}(1!d_1,2!d_2,\ldots), \label{eq:bell-form}
\end{align}
where \(\poch{r}{j}=r(r+1)\cdots(r+j-1)\) and \(B_{n,j}\) is the
partial exponential Bell polynomial.
\end{theorem}
\begin{proof}
Because \(D(z)\) has zero constant term, the generalized binomial
theorem applies formally:
\[
 (1-D(z))^{-r}
 =
 \sum_{j=0}^{\infty}
 \binom{r+j-1}{j}D(z)^j.
\]
For formal powers of series, see
\cite[Ch.~I, \S1.12; Ch.~III, \S3.5]{Comtet1974}.
For a fixed \(n\geq1\), the term \(j=0\) contributes only to the
constant coefficient, and \(D(z)^j\) begins in degree \(j\).  Hence
only \(1\leq j\leq n\) can contribute to \([z^n]\).
Multiplying the power series gives and by Cauchy product gives
\[
 \begin{aligned}
 D(z)^j
 &=
 \left(\sum_{m=1}^{\infty}d_mz^m\right)^j=
 \sum_{n=j}^{\infty}
 \left(
 \sum_{\substack{n_1+\cdots+n_j=n\\n_i\geq1}}
 d_{n_1}\cdots d_{n_j}
 \right)z^n.
 \end{aligned}
\]
Therefore,
\[
 [z^n]D(z)^j
 =
 \sum_{\substack{n_1+\cdots+n_j=n\\n_i\geq1}}
 d_{n_1}\cdots d_{n_j}.
\]
Substitution into the binomial expansion gives
\[
 c_n^{(r)}
 =
 \sum_{j=1}^{n}
 \binom{r+j-1}{j}
 \sum_{\substack{n_1+\cdots+n_j=n\\n_i\geq1}}
 d_{n_1}\cdots d_{n_j},
\]
which proves \eqref{eq:composition-form}.

For the Bell-polynomial form, recall that
\[
 B_{n,j}(x_1,x_2,\ldots)
 =
 \sum
 \frac{n!}{m_1!m_2!\cdots}
 \prod_{\ell\geq1}
 \left(\frac{x_\ell}{\ell!}\right)^{m_\ell},
\]
where the sum is taken over all nonnegative integers
\(m_1,m_2,\ldots\) satisfying
\[
 m_1+m_2+\cdots=j,
 \qquad
 m_1+2m_2+3m_3+\cdots=n.
\]
Their standard generating identity is
\[
 \frac1{j!}
 \left(
 \sum_{\ell=1}^{\infty}
 \frac{x_\ell}{\ell!}z^\ell
 \right)^j
 =
 \sum_{n=j}^{\infty}
 B_{n,j}(x_1,x_2,\ldots)\frac{z^n}{n!};
\]
see \cite[Ch.~III, Eq.~(3a')]{Comtet1974}.
Taking \(x_\ell=\ell!d_\ell\), we obtain
\[
 D(z)^j
 =
 j!\sum_{n=j}^{\infty}
 B_{n,j}(1!d_1,2!d_2,\ldots)\frac{z^n}{n!}.
\]
Hence
\[
 [z^n]D(z)^j
 =
 \frac{j!}{n!}
 B_{n,j}(1!d_1,2!d_2,\ldots).
\]
Using this coefficient identity together with
\[
 j!\binom{r+j-1}{j}
 =
 r(r+1)\cdots(r+j-1)
 =
 \poch{r}{j},
\]
gives \eqref{eq:bell-form}.
\end{proof}
The first coefficients, valid for either transform, are
\begin{align}
 c_0^{(r)}&=1, \nonumber\\
 c_1^{(r)}&=rd_1, \nonumber\\
 c_2^{(r)}&=rd_2+\frac{r(r+1)}2d_1^2, \label{eq:first-general}\\
 c_3^{(r)}&=rd_3+r(r+1)d_1d_2
   +\frac{r(r+1)(r+2)}6d_1^3. \nonumber
\end{align}
For each fixed degree \(n\), the coefficient \(c_n^{(r)}\) is a
polynomial in \(r\).

We next evaluate the reciprocal series and record quantitative bounds
for their convergence.

\begin{proposition}\label{prop:zero-free}
For real \(s>0\), both \(H_s(w)\) and \(A_s(u)\) are analytic and do not vanish in the unit disk.
\end{proposition}

\begin{proof}
For \(|w|<1\), the image point
\(
 t=\frac{w}{1-w}
\)
lies in \(\RePart t>-1/2\).  Likewise, for \(|u|<1\),
\(
 t=\frac{2u}{1-u}
\)
lies in \(\RePart t>-1\).  In both cases, \eqref{eq:F-positive-real} gives
\[
 \RePart H_s(w)>0,\qquad
 \RePart A_s(u)>0,
\]
respectively.  Hence neither function vanishes.
\end{proof}

We distinguish the reciprocal coefficients of the two transforms by
writing
\begin{align}
 H_s(w)^{-r}
 &=\sum_{n=0}^{\infty}b_{s,r,n}w^n, \label{eq:b-definition}\\
 A_s(u)^{-r}
 &=\sum_{n=0}^{\infty}c_{s,r,n}u^n. \label{eq:c-definition}
\end{align}
Thus, \(b_{s,r,n}\) is obtained from the general coefficient
\(c_n^{(r)}\) by taking \(d_n=q_{s,n}\), while \(c_{s,r,n}\) is
obtained by taking \(d_n=p_{s,n}\).  The indices record the beta parameter and the
reciprocal power; the letters distinguish the two transformations.

\begin{theorem}[Two positive reciprocal-beta series]
\label{thm:main-two-series}
Let \(r\geq1\) be an integer.
\begin{enumerate}[label=(\alph*)]
 \item For real \(s>0\),
 \begin{equation}\label{eq:beta-base2}
  \frac1{\beta(s)^r}
  =\sum_{n=0}^{\infty}\frac{b_{s,r,n}}{2^n},
 \end{equation}
 where the coefficients are obtained from
 \eqref{eq:recurrence} with \(d_n=q_{s,n}\).
 \item For real \(s\geq2\),
 \begin{equation}\label{eq:beta-base3}
  \frac1{\beta(s)^r}
  =\sum_{n=0}^{\infty}\frac{c_{s,r,n}}{3^n},
 \end{equation}
 where the coefficients are obtained from
 \eqref{eq:recurrence} with \(d_n=p_{s,n}\).
\end{enumerate}
All the displayed coefficients are strictly positive.  If \(s\) is
an integer, they are rational.
\end{theorem}

\begin{proof}
By Proposition~\ref{prop:zero-free}, the expansions
\eqref{eq:b-definition} and \eqref{eq:c-definition} converge
throughout the unit disk.  At the relevant evaluation points,
\[
 H_s(1/2)=F_s(1)=\beta(s),
\qquad
 A_s(1/3)=F_s(1)=\beta(s),
\]
so evaluation gives \eqref{eq:beta-base2} and
\eqref{eq:beta-base3}.  The positivity statements follow from
Theorems~\ref{thm:H-coefficients}, \ref{thm:p-positive}, and
\ref{thm:recurrence}.  When \(s\) is an integer, the finite formulas
\eqref{eq:q-finite} and \eqref{eq:p-finite} give rational base
coefficients, and the recurrence preserves rationality.
\end{proof}

\begin{lemma}\label{lem:tail}
Suppose
\[
C(z)=\sum_{n=0}^{\infty}a_nz^n,\qquad a_n\geq0,
\]
has radius of convergence at least \(1\).  If
\(0<x<\rho<1\) and \(N\geq0\) is an integer, then
\begin{equation}\label{eq:positive-tail}
 0\leq C(x)-\sum_{n=0}^{N}a_nx^n
 \leq
 \left(\frac{x}{\rho}\right)^{N+1}C(\rho).
\end{equation}
\end{lemma}
\begin{proof}
The assumptions ensure convergence at both \(x\) and \(\rho\), and
\[
 C(x)-\sum_{n=0}^{N}a_nx^n
 =
 \sum_{n=N+1}^{\infty}a_nx^n.
\]
The remainder is nonnegative.  For \(n\geq N+1\),
\[
 x^n
 =
 \rho^n\left(\frac{x}{\rho}\right)^n.
\]
Because \(0<x/\rho<1\),
\[
 \left(\frac{x}{\rho}\right)^n
 \leq
 \left(\frac{x}{\rho}\right)^{N+1},
\]
and hence
\(
 x^n
 \leq
 \rho^n\left(\frac{x}{\rho}\right)^{N+1}.
\)
Multiplying by \(a_n\) and summing gives
\[
 \begin{aligned}
 C(x)-\sum_{n=0}^{N}a_nx^n
 &=
 \sum_{n=N+1}^{\infty}a_nx^n\leq
 \left(\frac{x}{\rho}\right)^{N+1}
 \sum_{n=N+1}^{\infty}a_n\rho^n\leq
 \left(\frac{x}{\rho}\right)^{N+1}
 \sum_{n=0}^{\infty}a_n\rho^n\\
 &=
 \left(\frac{x}{\rho}\right)^{N+1}C(\rho).
 \end{aligned}
\]
This proves \eqref{eq:positive-tail}; the remainder is strict whenever
one of the omitted coefficients is positive.
\end{proof}

\begin{theorem}\label{thm:error-bounds}
Let \(r\geq1\) and \(N\geq0\) be integers.  For real \(s>0\), put
\[
 R_{s,r,N}^{(2)}
 =
 \frac1{\beta(s)^r}
 -
 \sum_{n=0}^{N}\frac{b_{s,r,n}}{2^n}.
\]
Then, for \(1/2<\rho<1\),
\begin{equation}\label{eq:error-H-exact}
 0<R_{s,r,N}^{(2)}
 \leq
 \left(\frac1{2\rho}\right)^{N+1}
 F_s\!\left(\frac{\rho}{1-\rho}\right)^{-r},
\end{equation}
and
\begin{equation}\label{eq:error-H-simple}
 R_{s,r,N}^{(2)}
 \leq
 (1-\rho)^{-r}
 \left(\frac1{2\rho}\right)^{N+1}.
\end{equation}
If \(s\geq2\), put
\[
 R_{s,r,N}^{(3)}
 =
 \frac1{\beta(s)^r}
 -
 \sum_{n=0}^{N}\frac{c_{s,r,n}}{3^n}.
\]
Then, for \(1/3<\rho<1\),
\begin{equation}\label{eq:error-A-exact}
 0<R_{s,r,N}^{(3)}
 \leq
 \left(\frac1{3\rho}\right)^{N+1}
 F_s\!\left(\frac{2\rho}{1-\rho}\right)^{-r}.
\end{equation}
Moreover,
\begin{equation}\label{eq:error-A-simple}
 R_{s,r,N}^{(3)}
 \leq
 \left(\frac{1+\rho}{1-\rho}\right)^r
 \left(\frac1{3\rho}\right)^{N+1}.
\end{equation}
\end{theorem}

\begin{proof}
Apply Lemma~\ref{lem:tail} to \eqref{eq:b-definition} with
\(x=1/2\), and to \eqref{eq:c-definition} with \(x=1/3\).  This gives
the exact bounds; strict positivity follows because every omitted
coefficient is positive.

For real \(t\geq0\), \eqref{eq:F-integral} gives
\[
 F_s(t)\geq\frac1{1+t}
 \frac1{\Gamma(s)}
 \int_0^1(-\log x)^{s-1}\dd x
 =\frac1{1+t}.
\]
At \(t=\rho/(1-\rho)\), the lower bound is \(1-\rho\); at
\(t=2\rho/(1-\rho)\), it is \((1-\rho)/(1+\rho)\).  These two choices
give \eqref{eq:error-H-simple} and \eqref{eq:error-A-simple}.
\end{proof}

For later use, two convenient choices of \(\rho\) are
\begin{align}
 R_{s,r,N}^{(2)}
 &\leq4^r\left(\frac23\right)^{N+1}
 &&(\rho=3/4), \label{eq:error-H-concrete}\\
 R_{s,r,N}^{(3)}
 &\leq3^r\left(\frac23\right)^{N+1}
 &&(\rho=1/2). \label{eq:error-A-concrete}
\end{align}
Alternatively, \(\rho=2/3\) in the accelerated estimate gives
\begin{equation}\label{eq:error-A-half}
 R_{s,r,N}^{(3)}
 \leq
 \frac{F_s(4)^{-r}}{2^{N+1}}
 \leq\frac{5^r}{2^{N+1}}.
\end{equation}
These bounds are intentionally simple; the parameter \(\rho\) may be
optimized for chosen values of \(N\), \(s\), and \(r\).

\begin{proposition}\label{prop:root-rate}
For fixed real \(s>0\) and integer \(r\geq1\),
\begin{equation}
 \limsup_{n\to\infty}
 \left(\frac{b_{s,r,n}}{2^n}\right)^{1/n}
 =\frac12. \label{eq:root-H}
\end{equation}
If \(s\geq2\), then also
\begin{equation}
 \limsup_{n\to\infty}
 \left(\frac{c_{s,r,n}}{3^n}\right)^{1/n}
 =\frac13. \label{eq:root-A}
\end{equation}
\end{proposition}

\begin{proof}
Proposition~\ref{prop:zero-free} gives analyticity in the open unit
disk.  Along the positive real axis, dominated convergence in
\eqref{eq:F-integral} shows that
\[
 F_s(t)\longrightarrow0\qquad(t\to+\infty).
\]
Thus, \(H_s(w)^{-r}\) and \(A_s(u)^{-r}\) become unbounded as
\(w\to1^-\) and \(u\to1^-\), respectively.  Their Taylor series
therefore have radius exactly \(1\).  The Cauchy--Hadamard formula yields
\[
 \limsup_{n\to\infty}|b_{s,r,n}|^{1/n}=1,
\]
and, when \(s\geq2\),
\[
 \limsup_{n\to\infty}|c_{s,r,n}|^{1/n}=1.
\]
Since the relevant coefficients are positive, the absolute values may
be dropped.  Division by \(2^n\) and \(3^n\) gives the stated root
rates.
\end{proof}

\section{Reciprocal powers of even beta values}
\label{sec:even-beta}

Applying the general construction at positive even integers produces a
single family that includes the Catalan case \(\beta(2)=\Gconst\).

\begin{theorem}[Reciprocal powers of even beta values]
\label{thm:all-even-beta}
For all integers \(m,r\geq1\),
\begin{equation}\label{eq:all-even-beta}
\begin{aligned}
 \frac1{\beta(2m)^r}
 &=1+
 \sum_{n=1}^{\infty}\frac1{3^n}
 \sum_{j=1}^{n}\binom{r+j-1}{j}\\
 &\qquad {}\times
 \sum_{\substack{n_1+\cdots+n_j=n\\n_i\geq1}}
 \prod_{\ell=1}^{j}
 \left[
 \sum_{k=1}^{n_\ell}
 \frac{(-1)^{k+1}2^k}{(2k+1)^{2m}}
 \binom{n_\ell-1}{k-1}
 \right].
\end{aligned}
\end{equation}
Every term of the series following the initial term \(1\) is a
positive rational number.
\end{theorem}

\begin{proof}
Set \(s=2m\) in \eqref{eq:p-finite}.  The resulting base coefficients
are
\[
 p_{2m,n}
 =
 \sum_{k=1}^{n}
 \frac{(-1)^{k+1}2^k}{(2k+1)^{2m}}
 \binom{n-1}{k-1}.
\]
Theorem~\ref{thm:p-positive} gives positivity, and integrality of
\(2m\) gives rationality.  With \(d_n=p_{2m,n}\),
Theorem~\ref{thm:nonrecursive} gives the coefficient of \(u^n\) in
\(A_{2m}(u)^{-r}\).  Since
\(
 A_{2m}(1/3)=F_{2m}(1)=\beta(2m)
\)
by Theorem~\ref{thm:main-two-series}(b), evaluation at \(u=1/3\)
gives \eqref{eq:all-even-beta}.
\end{proof}

\subsection{Catalan's constant}
\label{sec:catalan}

The inverse tangent integral is
\[
 \Ti_2(z)
 =
 \int_0^z\frac{\arctan t}{t}\,\dd t
 =
 \sum_{n=0}^{\infty}
 \frac{(-1)^nz^{2n+1}}{(2n+1)^2},
 \qquad 0\leq z\leq1.
\]
In particular, \(\Ti_2(1)=\Gconst\); see
\cite[Eqs, (1) and (3)]{Ramanujan1915}.  For \(m=1\), comparison of the two
defining power series gives
\(
 F_2(t)=\frac{\Ti_2(\sqrt t)}{\sqrt t},
\)
initially for \(0\leq t<1\).  Its power series removes the apparent
singularity at \(t=0\), where the quotient has value \(1\).  The
accelerated transformation therefore becomes
\begin{equation}\label{eq:transformed-Ti2}
 \frac{\Ti_2\!\left(\sqrt{2u/(1-u)}\right)}
 {\sqrt{2u/(1-u)}}
 =
 1-\sum_{n=1}^{\infty}
 \left(
 \sum_{k=1}^{n}
 \frac{(-1)^{k+1}2^k}{(2k+1)^2}
 \binom{n-1}{k-1}
 \right)u^n.
\end{equation}
The left-hand side is analytic at \(u=0\) and equals \(1\) there.  At
\(u=1/3\), its square-root argument is \(1\), so the value is
\(\Ti_2(1)=\Gconst\).  The Catalan specialization of
Theorem~\ref{thm:all-even-beta} is therefore as follows.

\begin{corollary}[All reciprocal powers of \(\Gconst\)]
\label{cor:all-G}
For every integer \(r\geq1\),
\begin{equation}\label{eq:all-G}
\begin{aligned}
 \frac1{\Gconst^r}
 &=1+
 \sum_{n=1}^{\infty}\frac1{3^n}
 \sum_{j=1}^{n}\binom{r+j-1}{j}\\
 &\qquad {}\times
 \sum_{\substack{n_1+\cdots+n_j=n\\n_i\geq1}}
 \prod_{\ell=1}^{j}
 \left[
 \sum_{k=1}^{n_\ell}
 \frac{(-1)^{k+1}2^k}{(2k+1)^2}
 \binom{n_\ell-1}{k-1}
 \right].
\end{aligned}
\end{equation}
Every term of the series following the initial term \(1\) is a
positive rational number.
\end{corollary}

\begin{proof}
Take \(m=1\) in Theorem~\ref{thm:all-even-beta} and use
\(\beta(2)=\Gconst\).
\end{proof}

Table~\ref{tab:p2} lists the first twelve coefficients obtained from
\eqref{eq:p-finite}.  Although the finite sum alternates, positivity is
guaranteed by Theorem~\ref{thm:p-positive}.

\begin{table}[ht]
\centering
\caption{The first accelerated Catalan coefficients
\(p_{2,n}\).}\label{tab:p2}
\small
\begin{tabular}{c l@{\qquad}c l}
\toprule
\(n\)&\(p_{2,n}\)&\(n\)&\(p_{2,n}\)\\
\midrule
1&\(2/9\)
&7&\(49156258/2029052025\)\\
2&\(14/225\)
&8&\(1437349042/83770862175\)\\
3&\(722/11025\)
&9&\(540857698558/30241281245175\)\\
4&\(3422/99225\)
&10&\(19381997114/1440061011675\)\\
5&\(434066/12006225\)
&11&\(14930206511746/1066509185246505\)\\
6&\(9407686/405810405\)
&12&\(133043356333714/12119422559619375\)\\
\bottomrule
\end{tabular}
\end{table}

Starting from these base coefficients, the recurrence
\eqref{eq:recurrence} produces the values in Table~\ref{tab:c-first}.
The entries are \(c_{2,r,n}\) from \eqref{eq:c-definition}; the actual
series term is \(c_{2,r,n}/3^n\).

\begin{table}[ht]
\centering
\caption{Initial reciprocal coefficients \(c_{2,r,n}\) for
\(r=1,2,3,4\).}\label{tab:c-first}
\small
\setlength{\tabcolsep}{3pt}
\renewcommand{\arraystretch}{1.25}
\begin{tabular}{c c c c c}
\toprule
\(n\)&\(r=1\)&\(r=2\)&\(r=3\)&\(r=4\)\\
\midrule
0&\(1\)&\(1\)&\(1\)&\(1\)\\
1&\(2/9\)&\(4/9\)&\(2/3\)&\(8/9\)\\
2&\(226/2025\)&\(184/675\)&\(326/675\)&\(1504/2025\)\\
3&\(92978/893025\)&\(230252/893025\)
 &\(421622/893025\)&\(676888/893025\)\\
4&\(15897874/200930625\)&\(43596272/200930625\)
 &\(4115114/9568125\)&\(29634688/40186125\)\\
\bottomrule
\end{tabular}
\end{table}

Thus, each row contributes its entry divided by \(3^n\).  For
\(r=1\) and \(r=2\), one further use of the recurrence gives the first
five nonconstant terms:
\begin{align*}
 \frac1{\Gconst}
 &=
 1+\sum_{n=1}^{\infty}\frac1{3^n}
 \sum_{j=1}^{n}
 \sum_{\substack{n_1+\cdots+n_j=n\\n_i\geq1}}
 \prod_{\ell=1}^{j}
 \left[
 \sum_{k=1}^{n_\ell}
 \frac{(-1)^{k+1}2^k}{(2k+1)^2}
 \binom{n_\ell-1}{k-1}
 \right]\\
 &=
 1+\frac2{27}
 +\frac{226}{18225}
 +\frac{92978}{24111675}
 +\frac{15897874}{16275380625}
 +\frac{16451880818}{53171668501875}
 +\cdots
 \approx 1.0917440637\ldots,\\
 \frac1{\Gconst^2}
 &=
 1+\sum_{n=1}^{\infty}\frac1{3^n}
 \sum_{j=1}^{n}(j+1)
 \sum_{\substack{n_1+\cdots+n_j=n\\n_i\geq1}}
 \prod_{\ell=1}^{j}
 \left[
 \sum_{k=1}^{n_\ell}
 \frac{(-1)^{k+1}2^k}{(2k+1)^2}
 \binom{n_\ell-1}{k-1}
 \right]\\
 &=
 1+\frac4{27}
 +\frac{184}{6075}
 +\frac{230252}{24111675}
 +\frac{43596272}{16275380625}
 +\frac{2175404068}{2531984214375}
 +\cdots
 \approx 1.19190510063\ldots.
\end{align*}
All terms in these expansions are positive, so every partial sum is a
strict lower bound.  The remainder is controlled by
Theorem~\ref{thm:error-bounds}, and Proposition~\ref{prop:root-rate}
gives the geometric rate.
\section{\texorpdfstring{Odd beta values and reciprocal powers of \(\pi\)}
{Odd beta values and reciprocal powers of pi}}
\label{sec:odd-beta}

Odd beta values are rational multiples of odd powers of \(\pi\).
Combining these classical evaluations with the preceding construction
therefore yields positive rational series for reciprocal powers of
\(\pi\).

Setting \(x=1/2\) in the Euler-polynomial generating function, using
the convention in
\cite[Eqs.~(23.1.1)--(23.1.3), p.~804]{AbramowitzStegun1964},
and replacing the generating variable by \(2z\), we obtain the
Euler-number convention
\begin{equation}\label{eq:euler-number-convention}
 \operatorname{sech}z
 =
 \frac{2}{\e^z+\e^{-z}}
 =
 \sum_{n=0}^{\infty}E_n\frac{z^n}{n!},
 \qquad |z|<\frac{\pi}{2}.
\end{equation}
With this convention,
\[
 E_0=1,\quad E_2=-1,\quad E_4=5,\quad E_6=-61,\quad
 E_8=1385,
\]
while \(E_{2m+1}=0\) for \(m\geq0\).  They also satisfy the recurrence
\begin{equation}\label{eq:euler-number-recurrence}
 \sum_{j=0}^{m}\binom{2m}{2j}E_{2j}=0,
 \qquad m\geq1,
\end{equation}
obtained by multiplying the power series for
\(\operatorname{sech}z\) and \(\cosh z\).  The even-indexed terms
alternate in sign:
\[
 (-1)^mE_{2m}=|E_{2m}|.
\]
See \cite[Eq.~(23.1.15), p.~805]{AbramowitzStegun1964}.
The required beta values follow from the Fourier expansion of the Euler
polynomials:
\[
 E_{2m}(x)
 =
 (-1)^m\frac{4(2m)!}{\pi^{2m+1}}
 \sum_{k=0}^{\infty}
 \frac{\sin((2k+1)\pi x)}{(2k+1)^{2m+1}},
 \qquad m\geq1,\quad 0\leq x\leq1.
\]
Together with
\(E_n(1/2)=2^{-n}E_n\), this expansion is recorded in
\cite[Eqs.~(23.1.18) and (23.1.21), p.~805]{AbramowitzStegun1964}.  Setting
\(x=1/2\) and using
\[
 \sin\left(\frac{(2k+1)\pi}{2}\right)=(-1)^k
\]
yields the classical formula
\begin{equation}\label{eq:odd-beta}
 \beta(2m+1)
 =
 \frac{(-1)^mE_{2m}}{4^{m+1}(2m)!}\,
 \pi^{2m+1},
 \qquad m\geq1.
\end{equation}
The same identity is also recorded directly in
\cite[Eq.~(23.2.22), p.~807]{AbramowitzStegun1964}.  The remaining case
\(m=0\) is
\(
 \beta(1)=\pi/4.
\)
The first few odd beta values are
\begin{equation}\label{eq:first-odd-beta-values}
 \beta(3)=\frac{\pi^3}{32},\qquad
 \beta(5)=\frac{5\pi^5}{1536},\qquad
 \beta(7)=\frac{61\pi^7}{184320},\qquad
 \beta(9)=\frac{1385\pi^9}{41287680}.
\end{equation}

Substitution into the general reciprocal series gives the following
explicit family.

\begin{theorem}[Reciprocal powers of odd beta values]
\label{thm:odd-beta-powers}
For all integers \(m,r\geq1\),
\begin{equation}\label{eq:pi-family}
\begin{aligned}
 \frac{\bigl[4^{m+1}(2m)!\bigr]^r}
 {|E_{2m}|^r\pi^{(2m+1)r}}
 &=
 1+\sum_{n=1}^{\infty}\frac1{3^n}
 \sum_{j=1}^{n}\binom{r+j-1}{j}\\
 &\qquad {}\times
 \sum_{\substack{n_1+\cdots+n_j=n\\n_i\geq1}}
 \prod_{\ell=1}^{j}
 \left[
 \sum_{k=1}^{n_\ell}
 \frac{(-1)^{k+1}2^k}{(2k+1)^{2m+1}}
 \binom{n_\ell-1}{k-1}
 \right].
\end{aligned}
\end{equation}
For every \(n\geq1\), the coefficient of \(3^{-n}\) is a positive
rational number.
\end{theorem}

\begin{proof}
Set \(s=2m+1\) in \eqref{eq:p-finite}.  The corresponding base
coefficients are
\[
 p_{2m+1,n}
 =
 \sum_{k=1}^{n}
 \frac{(-1)^{k+1}2^k}{(2k+1)^{2m+1}}
 \binom{n-1}{k-1}.
\]
Because \(2m+1\geq3\), Theorem~\ref{thm:p-positive} shows that these
coefficients are positive.  They are rational because \(2m+1\) is an
integer.

With \(d_n=p_{2m+1,n}\), Theorem~\ref{thm:nonrecursive} gives the
coefficient of \(u^n\) in \(A_{2m+1}(u)^{-r}\), while
Theorem~\ref{thm:main-two-series}(b) gives
\[
 A_{2m+1}(1/3)^{-r}
 =
 \frac1{\beta(2m+1)^r}.
\]
Equation \eqref{eq:odd-beta}, together with
\(
 (-1)^mE_{2m}=|E_{2m}|
\), gives
\[
 \frac1{\beta(2m+1)^r}
 =
 \frac{\bigl[4^{m+1}(2m)!\bigr]^r}
 {|E_{2m}|^r\pi^{(2m+1)r}}.
\]
Combining the two evaluations gives \eqref{eq:pi-family}.  Positivity
and rationality follow from the corresponding properties of the base
coefficients and the finite composition formula.
\end{proof}

Using the notation \(c_{s,r,n}\) from \eqref{eq:c-definition}, the
first four cases take a compact form.

\begin{corollary}\label{cor:first-four-odd-beta}
For every integer \(r\geq1\),
\begin{align}
 \frac{32^r}{\pi^{3r}}
 &=\sum_{n=0}^{\infty}\frac{c_{3,r,n}}{3^n},
 \label{eq:pi3}\\
 \frac{(1536/5)^r}{\pi^{5r}}
 &=\sum_{n=0}^{\infty}\frac{c_{5,r,n}}{3^n},
 \label{eq:pi5}\\
 \frac{(184320/61)^r}{\pi^{7r}}
 &=\sum_{n=0}^{\infty}\frac{c_{7,r,n}}{3^n},
 \label{eq:pi7}\\
 \frac{(41287680/1385)^r}{\pi^{9r}}
 &=\sum_{n=0}^{\infty}\frac{c_{9,r,n}}{3^n}.
 \label{eq:pi9}
\end{align}
Every coefficient on the right-hand sides is a positive rational
number.
\end{corollary}

\begin{proof}
Take \(m=1,2,3,4\), respectively, in
Theorem~\ref{thm:odd-beta-powers} and use the values in
\eqref{eq:first-odd-beta-values}.
\end{proof}

After a positive rational rescaling, the theorem therefore gives a
positive rational series for
\[
 \pi^{-(2m+1)r}
\]
for every \(m\geq1\) and \(r\geq1\).
For computation, the finite formula \eqref{eq:p-finite} and recurrence
\eqref{eq:recurrence} are more convenient than the composition sum.
They evaluate all the series in \eqref{eq:pi-family}, including
\eqref{eq:pi3}--\eqref{eq:pi9}, using rational arithmetic alone.

The case \(m=0\), namely \(\beta(1)=\pi/4\), must be treated with
the ordinary Euler transformation because accelerated positivity fails
at \(s=1\).

\begin{corollary}[The case \(\beta(1)=\pi/4\)]
\label{cor:pi1-ordinary}
For every integer \(r\geq1\),
\begin{equation}\label{eq:pi1-ordinary}
\begin{aligned}
 \frac{4^r}{\pi^r}
 &=
 1+\sum_{n=1}^{\infty}\frac1{2^n}
 \sum_{j=1}^{n}\binom{r+j-1}{j}\\
 &\qquad {}\times
 \sum_{\substack{n_1+\cdots+n_j=n\\n_i\geq1}}
 \prod_{\ell=1}^{j}
 \left[
 \sum_{k=1}^{n_\ell}
 \frac{(-1)^{k+1}}{2k+1}
 \binom{n_\ell-1}{k-1}
 \right].
\end{aligned}
\end{equation}
For every \(n\geq1\), the coefficient of \(2^{-n}\) is a positive
rational number.
\end{corollary}

\begin{proof}
Set \(s=1\) in \eqref{eq:q-finite} and substitute the resulting
coefficients \(d_n=q_{1,n}\) into
Theorem~\ref{thm:nonrecursive}.  Theorem~\ref{thm:main-two-series}(a)
then gives
\[
 H_1(1/2)^{-r}
 =
 \frac1{\beta(1)^r}
 =
 \frac{4^r}{\pi^r}.
\]
Positivity follows from Theorem~\ref{thm:H-coefficients} together with
Theorem~\ref{thm:nonrecursive}, and the finite coefficient formula
gives rationality.
\end{proof}
\section{Numerical comparison and verification}\label{sec:numerics}

For the Catalan case, write
\begin{equation}\label{eq:S-rN}
 S_{r,N}=\sum_{n=0}^{N}\frac{c_{2,r,n}}{3^n}.
\end{equation}
Table~\ref{tab:errors} reports the resulting absolute errors.  The
coefficients were kept exact, and \(\Gconst\) was evaluated
independently at high precision.

\begin{table}[ht]
\centering
\caption{Absolute error
\(\left|S_{r,N}-\Gconst^{-r}\right|\).}\label{tab:errors}
\small
\begin{tabular}{r c c c c}
\toprule
\(N\)&\(r=1\)&\(r=2\)&\(r=3\)&\(r=4\)\\
\midrule
5  &\(1.271{\times}10^{-4}\)&\(3.817{\times}10^{-4}\)
   &\(8.277{\times}10^{-4}\)&\(1.550{\times}10^{-3}\)\\
10 &\(3.796{\times}10^{-7}\)&\(1.296{\times}10^{-6}\)
   &\(3.184{\times}10^{-6}\)&\(6.724{\times}10^{-6}\)\\
20 &\(4.401{\times}10^{-12}\)&\(1.777{\times}10^{-11}\)
   &\(5.109{\times}10^{-11}\)&\(1.252{\times}10^{-10}\)\\
40 &\(8.494{\times}10^{-22}\)&\(4.145{\times}10^{-21}\)
   &\(1.425{\times}10^{-20}\)&\(4.147{\times}10^{-20}\)\\
80 &\(4.661{\times}10^{-41}\)&\(2.798{\times}10^{-40}\)
   &\(1.172{\times}10^{-39}\)&\(4.127{\times}10^{-39}\)\\
\bottomrule
\end{tabular}
\end{table}

For \(r=1\), Table~\ref{tab:compare} compares the ordinary
\(2^{-n}\)-series with the accelerated \(3^{-n}\)-series.  Proposition
\ref{prop:root-rate} identifies their exponential scales as \(1/2\)
and \(1/3\), but does not determine the constant in a term-by-term
comparison.  The numerical data suggest the following sharper
asymptotic statement.

\begin{conjecture}\label{conj:term-ratio}
For fixed integers \(s\geq2\) and \(r\geq1\),
\begin{equation}\label{eq:conjectured-term-ratio}
 \frac{c_{s,r,n}/3^n}{b_{s,r,n}/2^n}
 \sim
 2^{r/2}\left(\frac23\right)^n
 \qquad(n\to\infty).
\end{equation}
\end{conjecture}

The table confirms the faster convergence for \(s=2\) and \(r=1\),
although it is not designed to determine the constant in
Conjecture~\ref{conj:term-ratio}.

\begin{table}[ht]
\centering
\caption{Error comparison for \(1/\Gconst\).}\label{tab:compare}
\small
\begin{tabular}{r c c}
\toprule
\(N\)&ordinary Euler series&accelerated series\\
\midrule
5  &\(1.387{\times}10^{-3}\)&\(1.271{\times}10^{-4}\)\\
10 &\(3.168{\times}10^{-5}\)&\(3.796{\times}10^{-7}\)\\
20 &\(2.166{\times}10^{-8}\)&\(4.401{\times}10^{-12}\)\\
40 &\(1.407{\times}10^{-14}\)&\(8.494{\times}10^{-22}\)\\
80 &\(8.582{\times}10^{-27}\)&\(4.661{\times}10^{-41}\)\\
\bottomrule
\end{tabular}
\end{table}

The computation proceeds entirely with rational coefficients:
\(p_{2,n}\) comes from \eqref{eq:p-finite}, the recurrence
\eqref{eq:recurrence} produces \(c_{2,r,n}\), and
\eqref{eq:S-rN} gives the partial sums.  Up to degree \(N\), this uses
\(O(N^2)\) rational operations and stores \(O(N)\) coefficients.  As
an independent check on the final decimal errors, we used Ramanujan's
rapidly convergent identity
\begin{align}\label{eq:independent-G}
 \Gconst
 =
 \frac{\pi}{8}\log(2+\sqrt3)
 +\frac38\sum_{n=0}^{\infty}
 \frac1{(2n+1)^2\binom{2n}{n}}.
\end{align}
This identity follows from
\cite[Eqs.~(7)--(8)]{Ramanujan1915}; see also
Bradley \cite[Eq.~(2)]{Bradley1999}.
\appendix

\section{Proof of the standard Dirichlet-series identity}
\label{app:standard-L-proof}

\begin{proof}[Proof of Theorem~\ref{thm:standard-L}]
For \(s>1\), the identity \(\beta(s)=L(s,\chi_4)\) gives the Euler
product
\begin{equation}\label{eq:beta-euler-product}
 \beta(s)
 =\prod_p
  \left(1-\frac{\chi_4(p)}{p^s}\right)^{-1}.
\end{equation}
Absolute convergence for \(s>1\) follows from
\cite[Thm.~11.7]{Apostol1976}, so the product may be inverted and
raised to the positive integer power \(r\):
\[
 \frac1{\beta(s)^r}
 =
 \prod_p
 \left(1-\frac{\chi_4(p)}{p^s}\right)^r
 =
 \prod_p\sum_{a=0}^{r}
 (-1)^a\binom{r}{a}
 \frac{\chi_4(p)^a}{p^{as}}.
\]
The binomial theorem shows that the coefficient attached to \(p^a\)
is
\[
 (-1)^a\binom{r}{a}\chi_4(p)^a
 \qquad(0\leq a\leq r),
\]
and it is zero for \(a>r\).

The function \(\mu_r\) is multiplicative because Dirichlet convolution
preserves multiplicativity \cite[Thm.~2.14]{Apostol1976}.  At a fixed
prime \(p\), its generating function is especially simple.  Since
\(
 \mu(1)=1,\,
 \mu(p)=-1,\,
 \mu(p^a)=0\, (a\geq2),
\)
we have
\[
 \sum_{a=0}^{\infty}\mu_r(p^a)z^a
 =
 \left(
  \sum_{a=0}^{\infty}\mu(p^a)z^a
 \right)^r
 =
 (1-z)^r
 =
 \sum_{a=0}^{r}(-1)^a\binom{r}{a}z^a.
\]
Comparing coefficients gives \eqref{eq:mu-r-prime}.

For \(n=\prod_p p^{\nu_p(n)}\), multiplicativity of \(\chi_4\) and
\(\mu_r\) shows that the coefficient obtained from the Euler factors
is
\[
 \prod_p
 \chi_4(p)^{\nu_p(n)}
 \mu_r\!\left(p^{\nu_p(n)}\right)
 =
 \chi_4(n)\mu_r(n).
\]
To justify expansion of the product, note from \eqref{eq:mu-r} that
\[
 \mu_r(n)
 =
 \sum_{n_1n_2\cdots n_r=n}
 \mu(n_1)\mu(n_2)\cdots\mu(n_r).
\]
Because \(|\mu(n_j)|\leq1\),
\[
 |\mu_r(n)|
 \leq
 \sum_{n_1n_2\cdots n_r=n}1,
\]
and \(|\chi_4(n)|\leq1\).  Hence
\[
 \sum_{n=1}^{\infty}
 \frac{|\chi_4(n)\mu_r(n)|}{n^s}
 \leq
 \sum_{n_1,\ldots,n_r\geq1}
 \frac1{(n_1\cdots n_r)^s}
 =\zeta(s)^r<\infty,
\]
where \(\zeta(s)\) is the Riemann zeta function.
The product therefore expands as an absolutely convergent Dirichlet
series, and the coefficient calculation above gives
\eqref{eq:standard-L-series}.
\end{proof}

\section{An alternative recurrence for the accelerated Catalan base
coefficients}
\label{app:ode}

Let \(\theta_t=t\,\frac{\dd}{\dd t}\).  From the defining series,
\begin{equation}\label{eq:F2-ode}
 (2\theta_t+1)^2F_2(t)=\frac1{1+t}.
\end{equation}
Under \(t=2u/(1-u)\),
\(
 \theta_t=u(1-u)\frac{\dd}{\dd u}.
\)
If
\[
 A_2(u)=\sum_{n=0}^{\infty}a_nu^n,
 \qquad a_0=1,\qquad a_n=-p_{2,n}\ (n\geq1),
\]
then comparing coefficients in \eqref{eq:F2-ode} gives
\begin{equation}\label{eq:a-ode-recurrence}
 (2n+1)^2a_n
 -8n(n-1)a_{n-1}
 +4(n-1)(n-2)a_{n-2}
 =2(-1)^n,
 \qquad n\geq1,
\end{equation}
with the last term omitted when \(n=1\) (equivalently,
\(a_{-1}=0\)).
This provides an independent exact check on
\eqref{eq:p-finite}.

\section*{Declarations}

\subsection*{Funding}

The author declares that no funds, grants, or other financial support
were received during the preparation of this manuscript.

\end{document}